\documentclass[11pt]{amsart}

\usepackage{amsmath,amssymb,amsthm}
\usepackage[mathscr]{euscript}
\usepackage{color}

\usepackage[margin=1in]{geometry}

\theoremstyle{definition}
\newtheorem{theorem}{Theorem}[section]
\newtheorem{prop}[theorem]{Proposition}

\newtheorem{example}[theorem]{Example}

\newtheorem{definition}[theorem]{Definition}
\newtheorem{corollary}[theorem]{Corollary}

\newcommand{\conv}[1]{\mathrm{conv}\{#1\}}
\newcommand{\RR}{\mathbb{R}}
\newcommand{\ZZ}{\mathbb{Z}}

\newcommand{\polyp}{\mathcal{P}}
\newcommand{\polyq}{\mathcal{Q}}

\newcommand\commentout[1]{}

\begin{document}


\title{Ehrhart series, unimodality, and integrally closed reflexive polytopes}

\author{Benjamin Braun}
\address{Department of Mathematics\\
         University of Kentucky\\
         Lexington, KY 40506--0027}
\email{benjamin.braun@uky.edu}

\author{Robert Davis}
\address{Department of Mathematics\\
         University of Kentucky\\
         Lexington, KY 40506--0027}
\email{davis.robert@uky.edu}

\date{27 August 2014}

\thanks{
The first author is partially supported by the National Security Agency through award H98230-13-1-0240.
The second author is partially supported by a 2013-2014 Fulbright U.S. Student Fellowship. 
The authors thank Benjamin Nill for useful insights that contributed to the proof of Theorem~\ref{freesum}, and Akihiro Higashitani for his thoughtful comments.
}

\begin{abstract}
An interesting open problem in Ehrhart theory is to classify those lattice polytopes having a unimodal $h^*$-vector. 
Although various sufficient conditions have been found, necessary conditions remain a challenge. 
In this paper, we consider integrally closed reflexive simplices and discuss an operation that preserves reflexivity, integral closure, and unimodality of the $h^*$-vector, providing one explanation for why unimodality occurs in this setting.
We also discuss the failure of proving unimodality in this setting using weak Lefschetz elements.
\end{abstract}

\maketitle


\section{Introduction}\label{sec:intro}

For a lattice polytope $\polyp \subseteq \RR^n$ of dimension $d$, consider the counting function $|m\polyp \cap \ZZ^n|$, where $m\polyp$ is the $m$-th dilate of $\polyp$.
The \emph{Ehrhart series} of $\polyp$ is
\[
E_{\polyp}(t) := 1 + \sum_{m\in \ZZ_{\geq 1}} |m\polyp \cap \ZZ^n|t^m \, .
\]
Combining two well-known theorems due to Ehrhart \cite{Ehrhart} and Stanley \cite{StanleyDecompositions}, there exist values $h_0^*,\ldots,h_d^*\in \ZZ_{\geq 0}$ with $h_0^*=1$ such that
\[
E_\polyp(t)=\frac{\sum_{j=0}^dh_j^*t^j}{(1-t)^{d+1}} \, .
\]
We say the polynomial $h^*_\polyp(t):=\sum_{j=0}^dh_j^*t^j$ is the \emph{$h^*$-polynomial} of $\polyp$ (sometimes referred to as the $\delta$-polynomial of $\polyp$) and the vector of coefficients $h^*(\polyp)$ is the \emph{$h^*$-vector} of $\polyp$.
That $E_\polyp(t)$ is of this rational form with $h^*_\polyp(1)\neq 0$ is equivalent to $|m\polyp\cap \ZZ^n|$ being a polynomial function of $m$ of degree $d$; the non-negativity of the $h^*$-vector is an even stronger property.
The $h^*$-vector of a lattice polytope $\polyp$ is a fascinating partial invariant.
Obtaining a general understanding of $h^*$-vectors of lattice polytopes and their geometric/combinatorial implications is currently of great interest.  

Recent work has focused on determining when $h^*(\polyp)$ is unimodal, that is, when there exists some $k$ for which $h_0^* \leq \cdots \leq h_k^* \geq \cdots \geq h_d^*$.
One reason combinatorialists are interested in unimodality results is that their proofs often point to interesting and unexpected properties of combinatorial, geometric, and algebraic objects.
In particular, symmetric $h^*$-vectors play a key role in Ehrhart theory through their connection to reflexive polytopes, defined below.
There are many interesting techniques for studying symmetric unimodal sequences, using tools from analysis, Lie theory, algebraic geometry, etc \cite{stanleylogconcave}.

\begin{definition}
A lattice polytope $\polyp$ is called {\em reflexive} if $0 \in \polyp^{\circ}$ and its {\em (polar) dual}
\[
\polyp^{\Delta} := \{y \in \RR^n : x \cdot y \leq 1 \mathrm{~for~all~} x \in \polyp\}
\]
is also a lattice polytope.
A lattice translate of a reflexive polytope is also called reflexive.
\end{definition}
Reflexive polytopes have been the subject of a large amount of recent research \cite{BatyrevDualPolyhedra,BeckHosten,BeyHenkWills,BraunEhrhartFormulaReflexivePolytopes,HaaseMelnikov,HibiDualPolytopes,MustataPayne,Payne}.
It is known from work of Lagarias and Ziegler \cite{lagariasziegler} that there are only finitely many reflexive polytopes (up to unimodular equivalence) in each dimension, with one reflexive in dimension one, $16$ in dimension two, $4\, 319$ in dimension three, and $473\, 800\, 776$ in dimension four according to computations by Kreuzer and Skarke \cite{KreuzerSkarke00}.
The number of five-and-higher-dimensional reflexives is unknown.
One of the reasons reflexives are of interest is the following.
\begin{theorem}[Hibi, \cite{HibiDualPolytopes}] \label{SymmetricCoefficients}
A $d$-dimensional lattice polytope $\polyp\subset \RR^d$ containing the origin in its interior is reflexive if and only if $h^*(\polyp)$ satisfies $h_i^*=h_{d-i}^*$.
\end{theorem}

Hibi \cite{Hibi} conjectured that every reflexive polytope has a unimodal $h^*$-vector. 
Counterexamples to this were found in dimensions 6 and higher by Musta{\c{t}}{\v{a}} and Payne \cite{MustataPayne,Payne}.
However, Hibi and Ohsugi \cite{hibiohsugiconj} also asked whether or not every normal reflexive polytope has a unimodal $h^*$-vector; we consider the related question for integrally closed reflexives, where integral closure is defined as follows.
\begin{definition}
A lattice polytope $\polyp \subseteq \RR^n$ is {\em integrally closed} if, for every $x \in m\polyp \cap \ZZ^n$, 
there exist $x_1,\ldots,x_m \in \polyp \cap \ZZ^n$ such that $x = x_1 + \cdots + x_m$. 
\end{definition}
While the terms integrally closed and normal are often used interchangably, these are not synonymous \cite{gubeladzeconvexnormality}.
The counterexamples found by Musta{\c{t}}{\v{a}} and Payne are not normal, hence not integrally closed.
It remains to be seen whether or not every integrally closed reflexive polytope has a unimodal $h^*$-vector.
A stronger open question is whether or not being integrally closed is alone sufficient to imply unimodality \cite{schepersvanl}.
One condition that forces a lattice polytope $\polyp$ to be integrally closed is if $\polyp$ admits a unimodular triangulation; the latter condition has been shown to imply unimodality in the reflexive case by Athanasiadis \cite{Athanasiadisbirkhoff} and Bruns and R\"omer \cite{BrunsRomer}.

The purpose of this note is to investigate reflexive simplices in this context.
Several interesting recent results and counterexamples in Ehrhart theory have involved only simplices \cite{hibihermite,higashitanicounterexamples,MustataPayne,Payne}, so this is a reasonable restriction to make.
Our first main observation, Corollary~\ref{cor:main}, is that one can produce new reflexive, integrally closed simplices with unimodal $h^*$-vectors from two simplices having these three properties.
This provides an explanation for the presence of unimodal $h^*$-vectors among some reflexive, integrally closed simplices.
Our second main observation is that reflexive simplices that decompose as free sums are detectable by studying the \emph{type vector} of the simplex, as we explain in Theorem~\ref{freesum}.
If one wishes to search for an example of an integrally closed reflexive simplex with a non-unimodal $h^*$-vector, this allows a more refined search.
Our final observation is that in the case of a reflexive simplex $\polyp$, one might hope to study unimodality of the $h^*$-vector through algebraic methods applied to the semigroup algebra associated to $\polyp$.
We show in Proposition~\ref{nolefschetz} that there exist reflexive polytopes in all dimensions greater than two for which standard methods fail, specifically that weak Lefschetz elements need not exist in the quotient of the semigroup algebra by a system of parameters.


\section{Free Sums of Reflexive Simplices}\label{sec:freesum}

We follow the notation of \cite{BJM13} and define the relevant operation on polytopes that we will consider.

\begin{definition}	
	Suppose $\polyp, \polyq \subseteq \RR^n$ are lattice polytopes. 
        Call $\polyp \oplus \polyq := \conv{\polyp \cup \polyq}$ a {\em free sum} if, up to unimodular equivalence, $\polyp \cap \polyq = \{0\}$ and the affine spans of $\polyp$ and $\polyq$ are orthogonal coordinate subspaces of $\RR^n$. 
\end{definition}

\begin{example}
	The Reeve tetrahedron, $\mathcal{R}_h = \conv{0, e_1, e_2, e_1+e_2+he_3} \subseteq \RR^3$, $h > 1$ an integer, {\em cannot} be expressed as a free sum; if it could, then the lattice generated by $\mathcal{R}_h$ would be $\ZZ^3$. However, it only generates $\ZZ^2 \times h\ZZ$.
\end{example}

\begin{example}
The $d$-cross-polytope, given by $\conv{e_1,\ldots,e_d,-e_1,\ldots,-e_d}\subset \RR^d$, is a $d$-fold free sum of $[-1,1]$.
\end{example}

As with normality and integral closure, one must be cautious when discussing free sums; different authors sometimes use different definitions, and the validity of 
results may change based on which definition is used. The definition above is useful due to the following result.

\begin{theorem}\cite[Corollary 3.4]{BJM13}\label{HStarFreeSum}
	If $\polyp, \polyq \subseteq \RR^n$ are reflexive polytopes such that $0 \in \polyp^{\circ}$ and $\polyp \oplus \polyq = \conv{\polyp \cup \polyq}$ is a free sum, then
	$$h^*_{\polyp \oplus \polyq}(t) = h^*_{\polyp}(t)h^*_{\polyq}(t).$$
\end{theorem} 

Our next proposition provides a method for producing reflexive simplices from pairs of lower-dimensional reflexive simplices.

\begin{prop}\label{StarConstruction}
	Suppose $\polyp \subseteq \RR^n$ and $\polyq \subseteq \RR^m$ are full-dimensional simplices with $0 \in \polyp$ and $\{v_0, \ldots, v_m\}$ denoting the vertices
	of $\polyq$. Then for each $i = 0, 1, \ldots, m$ the polytope formed by
	$$\polyp *_i \polyq := \conv{(\polyp \times 0^m) \cup (0^n \times \polyq - v_i)} \subseteq \RR^{n+m}$$
	is a free sum and is itself a simplex. Moreover, if $0 \in \polyp^{\circ}$ and $\polyp$ and $\polyq$ are both reflexive, then $\polyp *_i \polyq$ is also reflexive.
\end{prop}

\begin{proof}
	Since each of $\polyp$ and $\polyq - v_i$ are full-dimensional, their affine spans are orthogonal subspaces of $\RR^{n+m}$. 
	Moreover, their intersection is 0, so the operation gives a free sum. 
	By Theorem~\ref{HStarFreeSum}, the denominator of $E_{\polyp *_i \polyq}(t)$ as a rational function is of degree $n + m + 1$, so $\dim(\polyp *_i \polyq) = n + m$. 
	Because $\polyp *_i \polyq$ is the convex hull of $n+m+2$ distinct point, but one vertex lies inside $\polyp \times 0^m$, it can be expressed as a convex hull
	of at most $n+m+1$ points. Thus $\polyp *_i \polyq$ is a simplex.
	
	Now we assume that both $\polyp$ and $\polyq$ are reflexive. 
	Noting that $E_{\polyq - v_i}(t) = E_{\polyq}(t)$, Theorem~\ref{HStarFreeSum} tells us that the numerator of $E_{\polyp *_i \polyq}(t)$ as a rational function has degree $n + m$. 
	This polynomial also has symmetric coefficients, since it is the product of polynomials that each have symmetric coefficients.
	A well-known result in Ehrhart theory tells us that the smallest dilate of $\polyp *_i \polyq$ containing an interior lattice point is $\dim (\polyp *_i \polyq) - (n + m - 1) = 1$. 
        Thus, by Theorem~\ref{SymmetricCoefficients}, the constructed simplex must be reflexive. 
\end{proof}

Geometrically, applying this operation to reflexive simplices corresponds to fixing $\polyp$ and translating $\polyq$ so that their intersection point is a vertex of $\polyq$ and the unique interior point of $\polyp$. 

An important property of the $*_i$ operation is that, under appropriate constraints, it preserves being integrally closed.

\begin{theorem}\label{integrallyclosed}
	If $\polyp$ and $\polyq$ are any integrally closed simplices with $0 \in \polyp^{\circ}$ and $\polyp$ reflexive, then $\polyp *_i \polyq$ is integrally closed.
\end{theorem}

\begin{proof}
	Since $\polyp *_i \polyq$ is a free sum, we may assume that $\polyp$ and $\polyq$ intersect at the origin 
	with $\polyp \subseteq \RR^n \times 0^m$ and $\polyq \subseteq 0^n \times \RR^m$. 
	
	By definition, the convex hull of $\polyp$ and $\polyq$ is the set of points representable as
	$$\sum_{i=1}^r \alpha_ip_i + \sum_{j=1}^s\beta_jq_j$$
	where $p_i \in \polyp, q_j \in \polyq$ for each $i,j$, and the $\alpha_i, \beta_j$ are nonnegative numbers whose total sum is 1. Form the points
	$$u = \frac{1}{\sum_{j=1}^r \alpha_j}\left(\sum_{i=1}^r \alpha_ip_i\right), v = \frac{1}{\sum_{k=1}^s \beta_k}\left(\sum_{l=1}^s \beta_lq_l\right).$$
	Then $u \in \polyp$ and $v \in \polyq$. Setting $t = \sum_{i=1}^r \alpha_i$, their convex sum
	$$\left(\sum_{i=1}^r \alpha_i\right)u + \left(\sum_{j=1}^s \beta_j\right)v = tu + (1-t)v$$
	is in $\polyp \oplus \polyq$, and, in particular, is in $t\polyp \times (1-t)\polyq$. Therefore the free sum is covered by sets of this form for $0 \leq t \leq 1$. 

	For the last step, let $(p,q) \in t\polyp \times (m-t)\polyq$ where $m$ is a positive integer, $p \in t\polyp$, and $q \in (m-t)\polyq$. Since $\polyp$ is reflexive, $p$ lies
	on the boundary of some integer scaling of $\polyp$, thus we may assume $t$ is an integer. Hence $q$ is in an integer scaling of $\polyq$. By the integral closure 
	of $\polyp$ and $\polyq$, there are $t$ lattice points of $\polyp$ summing to $x$ and $m-t$ lattice points of $\polyp$ summing to $y$. These summands are 
	all contained in $\polyp *_i \polyq$, hence it is integrally closed.
\end{proof}

This brings us to our main observation.

\begin{corollary}\label{cor:main}
	If $\polyp$ and $\polyq$ are integrally closed, reflexive simplices with $0\in \polyp^\circ$, then so is $\polyp *_i \polyq$ for each $i$. 
        If, in addition, $h^*(\polyp)$ and $h^*(\polyq)$ are unimodal, then so is $h^*(\polyp *_i \polyq)$.
\end{corollary}

\begin{proof}
	Integral closure follows from Theorem~\ref{integrallyclosed}, and reflexivity follows from Proposition~\ref{StarConstruction}. 
	By Theorem~\ref{HStarFreeSum} and \cite[Proposition 1]{stanleylogconcave}, which states that the product of two polynomials with symmetric unimodal coefficients has these same properties, the last claim holds.
\end{proof}

We end this section by noting that the conclusions of Proposition~\ref{StarConstruction} and Theorem~\ref{integrallyclosed} still hold when ``simplex'' is replaced with ``polytope;''
adaptations of their proofs are straightforward. 
However, there is no classification for arbitrary reflexive polytopes by type in the manner that we discuss in the next section.
Regardless, this gives one reason why a reflexive polytope may have a unimodal $h^*$-vector.
We remark that it is not clear what the relationship is between polytopes formed when using the $*_i$ construction on different vertices of the second operand, and it is not easy to identify geometrically that a reflexive polytope decomposes as a free sum.


\section{Searching for non-unimodal examples}

If one wishes to search for an example of an integrally closed, reflexive polytope with a non-unimodal $h^*$-vector, then it is natural to first rule out those polytopes obtained as a result of Corollary~\ref{cor:main}.
As mentioned in the introduction, reflexive simplices are a class one might focus on when searching for such a polytope.
It is helpful in this case to consider how an algorithm for producing all reflexive simplices, due to Conrads \cite{conrads}, interacts with the free sum operation.

The algorithm assigns to each reflexive simplex a {\em type} in the following way: let $v_0,\ldots,v_n$ be an ordering of its vertices, and construct 
$Q=(q_0,\ldots, q_n)$ by setting
	\[q_i = \left|\det \begin{pmatrix}
			| & | & \cdots & | & \cdots & | \\
			v_0 & v_1 & \cdots & \widehat{v_i} & \cdots & v_n \\
			| & | & \cdots & | & \cdots & |\\
		\end{pmatrix}\right| \, .
        \]
Note that reordering $Q$ corresponds to performing this same process to a unimodularly equivalent simplex. 
Thus, we may assume that $Q$ is nondecreasing.
Setting $\lambda = \gcd(q_0,\ldots, q_n)$ and $Q_{red}=\frac{1}{\lambda}Q$, we say the reflexive simplex has \emph{type} $(Q_{red},\lambda)$. 
We note that the simplices of type $(Q_{red},1)$ are
exactly those such that $Q=(q_0,\ldots,q_n)$ is a sequence of positive, nondecreasing integers where
	\begin{eqnarray}\label{conrads}
		\gcd(q_0,\ldots,q_n) = 1 \mathrm{~and~} q_i \mathrm{~divides~} \sum_{j=0}^n q_j \mathrm{~for~each~}i\in\{0,\ldots,n\} \, .
	\end{eqnarray}
In this case, each of these vectors corresponds to a unique reflexive simplex, which we denote $\Delta_Q$. 
To construct the reflexive simplices of a fixed dimension, we first construct all $Q$ satisfying \eqref{conrads} and form the corresponding $\Delta_Q$. 
The remaining simplices in this dimension are found by performing various additional operations on the $\Delta_Q$ \cite{conrads}.

For any reflexive simplex, we call $Q_{red}$ the {\em reduced weight} of the simplex, and a simplex with this reduced weight has the property that 
\[
\sum_i \frac{q_i}{\sum_\beta q_\beta}v_i = 0 \, .
\]
This follows from scaling the equality 
\[
\sum_i q_iv_i = 0\, ,
\]
which itself follows from Cramer's rule.
Note that because there are $n+1$ of the $v_i$'s in $n$-dimensional space, the coefficients of the above sum are uniquely determined up to scaling.
Thus, the $Q_{red}$ vector of a reflexive simplex is the particular choice of coefficients for this sum that satisfies the divisibility condition \eqref{conrads}.

\begin{example}
	The weight $Q=(1,1,1,\ldots,1) \in \ZZ^{n+1}$ corresponds to the polytope 
\[
\Delta_Q=\conv{e_1,\ldots,e_n,-\sum_ie_i} \, ,
\]
which is often called the \emph{standard reflexive simplex of minimal volume}.
Note that the sum of these vertices, each weighted by $1$, is equal to zero.
\commentout{
       This polytope is indeed integrally closed: applying $\frac{k}{n+1}$ for $k=0,1,\ldots,n$ results in an integer point contained in the fundamental parallelepiped of 
	$\mathrm{cn}(1,\Delta_Q)$, and the determinant 
	$$\det
	\begin{pmatrix}
		1 & 1 & \cdots & 1 & 1\\
		| & | & \cdots & | & -1 \\
		e_1 & e_2 & \cdots & e_n & \vdots \\
		| & | & \cdots & | & -1
	\end{pmatrix} = n+1$$
	shows that these are the only integer points in the fundamental parallelepiped. 
        Since they are obtained by summing $(1, 0, \ldots, 0)$	$k$ times, it follows that $\mathrm{cn}(1,\Delta_Q) \cap \ZZ^{n+1}$ is generated by its height 1 elements, giving us that $\Delta_Q$ is integrally closed.
}
	It is well known that one can demonstrate that this polytope is integrally closed by showing that it has a unimodular triangulation, specifically the triangulation whose facets consist of those simplices that are the convex hull of the origin and all but one of the vertices of $\Delta_Q$.
\end{example}

This $*_i$ operation has a corresponding interpretation in terms of the types of the summands.

\begin{theorem} \label{freesum}
	If $\polyp = \conv{v_0,\ldots,v_n} \subseteq \RR^n$ and $\polyq = \conv{w_0,\ldots,w_m} \subseteq \RR^m$ are full-dimensional reflexive simplices of types $((p_0,\ldots,p_n),\lambda)$ and $((q_0,\ldots,q_m),\mu)$, respectively, then $\polyp *_i \polyq$ is a reflexive simplex of type
	\[
        \left(\frac{1}{d}(q_ip_0,q_ip_1,\ldots,q_ip_n,sq_0,sq_1,\ldots,\widehat{sq_i},\ldots,sq_m),d\right),
        \]
	where $s = \sum_{j=0}^n p_j$ and $d = \gcd(q_i,\sum_{j=0}^n p_j)$. 
\end{theorem}

\begin{proof}
	For notational convenience, we identify $\polyp$ and $\polyq$ with their embeddings in $\RR^{n+m}$.
	Before the embedding, we know from the weights of $\polyp$ and $\polyq$ that 
	\[
        \sum_{j=0}^n \frac{p_j}{\sum{p_{\alpha}}}v_j = 0 \mathrm{~and~} \sum_{k=0}^m \frac{q_k}{\sum q_{\beta}}w_k = 0 \, .
        \]
	After the embedding, the translation of $\polyq$ in $\RR^{n+m}$ results in
	\[
        \sum_{k=0}^m \frac{q_k}{\sum{q_{\beta}}}(w_k - w_i) = -w_i \, .
        \]
	Therefore, on the vertices of the free sum, we see
	\begin{eqnarray*}
		-w_i &=& \frac{q_i}{\sum q_{\beta}}(w_i - w_i) + \sum_{\substack{k=0\\ k \neq i}}^m \frac{q_k}{\sum q_{\beta}}(w_k - w_i)\\
			&=& \sum_{j=0}^n \left(\frac{q_i}{\sum q_{\beta}}\cdot\frac{p_j}{\sum{p_{\alpha}}}\right)v_j +  \sum_{\substack{k=0\\ k \neq i}}^m \frac{q_k}{\sum q_{\beta}}(w_k-w_i) \, ,
	\end{eqnarray*}
	giving us the unique interior point of the simplex. 
        Thus, $Q_{red}$ for $\polyp *_i \polyq$ is given by a scaling of the vector 
        \[
        \left(\frac{q_i}{\sum q_{\beta}}\cdot\frac{p_0}{\sum{p_{\alpha}}},\frac{q_i}{\sum q_{\beta}}\cdot\frac{p_1}{\sum{p_{\alpha}}},\ldots,\frac{q_i}{\sum q_{\beta}}\cdot\frac{p_n}{\sum{p_{\alpha}}},\frac{q_0}{\sum q_{\beta}},\frac{q_1}{\sum q_{\beta}},\ldots,\widehat{\frac{q_i}{\sum q_{\beta}}},\ldots,\frac{q_m}{\sum q_{\beta}}\right) \, .
        \]
        Scaling this vector by $\left(\sum p_{\alpha}\right)\left(\sum q_\beta\right)$ and dividing by $\gcd(q_i,\sum_{j=0}^n p_j)$, we obtain an integer vector that satisfies \eqref{conrads}.
        Thus, this is our desired $Q_{red}$.
        To find the full $Q$ vector for $\polyp *_i \polyq$, we first translate the polytope by $w_i$ so that the interior vertex is zero, then compute determinants as described at the beginning of the section.
        Since the determinant of the matrix formed by $v_1+w_i,v_2+w_i,\ldots,v_n+w_i,w_0,w_1,\ldots,\widehat{w_i},\ldots,w_m$ (where all vectors are considered to be embedded in $\RR^{n+m}$) is equal to $q_ip_0$, this determines the type vector for $\polyp *_i \polyq$, and completes our proof.
\end{proof}

Thus, one way to search for examples of integrally closed reflexive simplices with non-unimodal $h^*$-vectors is to generate $Q$-vectors for the polytopes, then reduce the $Q$-vectors under consideration using Theorem~\ref{freesum} before testing $\Delta_Q$ for integral closure and unimodality.
This operation is particularly helpful when a simplex has type $(Q_{red},1)$, since it is the only simplex of that type. For example, $\Delta_{(1,1,2)}$ can be
decomposed as $\Delta_{(1,1)} *_0 \Delta_{(1,1)}$, since we know the $*_0$ operation provides a reflexive simplex of type $((1,1,2),1)$, and there is only one of this type. 
However, there may be multiple simplices of type $(Q_{red},\lambda)$ when $\lambda>1$, no longer guaranteeing that a simplex decomposes in a particular way. An 
example would be $((1,2,3,3,9),2)$;
there are two simplices of this type, but only one of them can be of the form $\Delta_{(1,2,3)} *_1 \Delta_{(1,2,3)}$. In this case, more checks 
are needed to identify which simplex decomposes as a free sum.

Unfortunately, while the free sum operation produces a large number of reflexive polytopes, it appears that these might be rare among the reflexive polytopes with unimodal $h^*$-vectors.
For example, when we randomly generated $1100$ eight-dimensional integrally-closed reflexive simplices, all of them had unimodal $h^*$-vectors, yet none of their type vectors split in the manner given in Theorem~\ref{freesum}.
It would be interesting to know more about the reflexive simplices formed via the free sum operation in comparision to the family of all reflexive simplices.


\section{The non-existence of Lefschetz elements}\label{lefschetzsec}

In this section, we show that a natural approach inspired by commutative algebra fails to establish unimodality for integrally closed reflexive simplices in general.
For any lattice simplex $\polyp \subseteq \RR^n$ with vertices $\{v_0,\ldots,v_n\}$, recall that there is an associated semigroup algebra given by
	\[
	\mathbb{C}[\polyp] := \mathbb{C}[x^az^m | a \in m\polyp\cap\ZZ^n]
	\]
where $x^a := x_1^{a_1}x_2^{a_2}\cdots x_n^{a_n}$. 
The Ehrhart series $E_{\polyp}(t)$ coincides with the Hilbert series of $\mathbb{C}[\polyp]$. 
In the case of a simplex $\polyp$ with vertices $v_0,\ldots,v_n$, it is straightforward to show that $h_k^*$ is equal to the number of lattice points satisfying $\sum c_i = k$ in the {\em fundamental parallelepiped} $\Pi(\polyp)$ defined by
	\[
	\Pi(\polyp) := \left\{\sum_{i=0}^n c_i(v_i,1) \Big | 0 \leq c_i < 1\right\} \subset \RR^{n+1} \; .
	\]
 This motivates us to study the zero-(Krull)-dimensional algebra
	\[
	R_{\polyp} := \mathbb{C}[\polyp] / (x^{v_0}z,\ldots,x^{v_n}z) \; ,
	\]
graded by the exponent on $z$.
The study of Hilbert functions gives a method for establishing unimodality of $h^*(\polyp)$ in this context.

\begin{definition}
	A linear form $l \in R_{\polyp}$ is called a {\em weak Lefschetz element} if the multiplication map
	\[
	\times l : [R_{\polyp}]_i \to [R_{\polyp}]_{i+1}
	\]
	has maximal rank, that is, is either injective or surjective, for each $i$.
\end{definition}

By Remark 3.3 of \cite{HarimaLefschetz}, if $R_{\polyp}$ has a weak Lefschetz element, then the Hilbert series has unimodal coefficients in its
numerator, and therefore so does $E_{\polyp}(t)$. Experimental data suggests that a weak Lefschetz element exists for many instances of $R_{\polyp}$
when $\polyp$ is an integrally closed reflexive simplex, but we will now show that such an element need not exist.

\begin{prop}\label{nolefschetz}
	For every $d \geq 3$, there exists a $d$-dimensional integrally closed reflexive simplex $\Delta_Q$ such that $R_{\Delta_Q}$ does not admit a weak Lefschetz element.
\end{prop}

\begin{proof}
	For fixed $d$, let $Q = (1, d, \underbrace{d+1, \ldots, d+1}_{d-1 \mathrm{~times~}})$. Then $Q$ defines a reflexive simplex
	\[
		\Delta_Q = \conv{e_1, \ldots, e_d, (-d, -d-1, \ldots, -d-1)^T)}. 
	\]
	Consider the cone consisting of all rays from the origin though a point in $(\Delta_Q,1) \subseteq \RR^{d+1}$.
	Elements of this cone with last coordinate $m$ are in bijection with points of $m(\Delta_Q)$ by projection onto the first $d$ coordinates.
	Additionally, the cone has hyperplane description given by $Ax \geq 0$, where $A$ is the $(d+1) \times (d+1)$ matrix
	\[
		\frac{1}{d(d+1)}
		\begin{pmatrix}
			d^2 & -d & -d & \cdots & -d & d \\
			-d-1 & d^2-1 & -d-1 & \cdots & -d-1 & d+1\\
			-d-1 & -d-1 & d^2-1 & \cdots & -d-1 & d+1 \\
			\vdots & \vdots & \vdots & \ddots & \vdots & \vdots \\
			-d-1 & -d-1 & -d-1 & \cdots & d^2-1 & d+1 \\
			-1 & -1 & -1 & \cdots & -1 & 1\\
		\end{pmatrix}.
	\]
Thus, there are $d(d+1)$ lattice points in $\Pi(\Delta_Q)$. 
	For each $r \in \{1, \ldots, d-1\}$, form the vectors
	\begin{eqnarray*}
			v_{0,r} &=& (0, 0, \ldots, 0, r)\\ 
			v_{1,r} &=& (-1, -1, \ldots, -1, r)\\
			\vdots && \vdots  \\
			v_{r-1,r} &=& (-r+1, -r+1, \ldots,-r+1, r)\\
			v_{r,r} &=& (-r, -r, \ldots, -r, r)\\
			v_{r+1,r} &=& (-r+1, -r, \ldots,-r, r)\\
			v_{r+2,r} &=& (-r, -r-1, \ldots, -r-1, r) \\
			\vdots && \vdots \\
			v_{d,r} &=& (-d+2, -d+1, \ldots, -d+1, r)\\
			v_{d+1,r} &=& (-d+1, -d, \ldots, -d, r)
	\end{eqnarray*}
	There are $d+2$ of these for each $r$, and along with the zero vector and $(-d+1,-d, \ldots, -d, d)$ we have $(d-1)(d+2) + 2 = d(d+1)$ total vectors, which we claim to be
	all of the lattice points in $\Pi(\Delta_Q)$.

To make this easier, we first show that $\Delta_Q$ is integrally closed.
Observe that every vector $v_{i,r}$ can be written as a sum of vectors $v_{j,r-1} + v_{k,1}$ in the following way.
We assume $r \geq 2$. 
When $i > r$, we may let $j = i$ and $k=0$; when $i < r$ we may let $j = i$ and $k = 1$; when $i = r$ we may use $j = r-1$ and $k = 1$.
Thus, by induction on $r$, every lattice point in $\Pi(\Delta_Q)$ is a sum of elements satisfying $r=1$.
Since every lattice point in the cone over $\Delta_Q$ is a sum of lattice points that are either ray generators or fundamental parallelepiped points, we conclude that $\Delta_Q$ is integrally closed.
	
	To see that the lattice points of $\Pi(\Delta_Q)$ are precisely those described above, we  show that all are 	obtained as a linear combination of the ray generators with coefficients less than one. 
	It is straightforward using the matrix above to show that all coefficients of the ray generators are less than $\frac{1}{d}$ when representing the lattice points in $\Pi(\Delta_Q)$ with $r=1$; integral closure
	then ensures that the coefficients of all points for $r \in \{2,\ldots,d-1\}$ will be bounded by $\frac{d-1}{d}$.
	Then one only needs to check the coefficients on the vector $(-d+1, -d, \ldots, -d, d)$.
	These verifications are also straightforward, and the details are omitted. 
	
	Now we must verify that no potential weak Lefschetz element is injective from $[R_{\Delta_Q}]_1$ to $[R_{\Delta_Q}]_2$.
	A weak Lefschetz element would be of the form
	\[
		\sum_{i=0}^{d+1} a_ix^{v_{i,1}}z
	\]
	where $a_i$ are field elements. 
        The map from $[R_{\Delta_Q}]_1$ to $[R_{\Delta_Q}]_2$ induced by multiplication by this element is representable as the matrix 
	\[
	\begin{pmatrix}
		a_0 & 0     & 0      & 0     & 0      &\cdots & 0 & 0 & 0 \\
		a_1 & a_0 & 0     & 0      & 0     & \cdots & 0 & 0 & 0 \\
		0      & a_1 & 0     & 0     & 0      & \cdots & 0 & 0 & 0 \\
		a_3 & a_2 & a_1 & a_0 & 0     & \cdots & 0 & 0 & 0 \\
		a_4 & a_3 & 0     & a_1  & a_0 &\cdots & 0 & 0 & 0  \\
		a_5 & a_4 & 0     & 0     & a_1  & \cdots & 0 & 0 & 0 \\
		\vdots & \vdots & \vdots & \vdots & \vdots & \ddots & \vdots & \vdots & \vdots \\
		a_d & a_{d-1} & 0 & 0 & 0 & \cdots & a_1 &a_0 &0\\
		a_{d+1} & a_d& 0 & 0 & 0 & \cdots &0 &a_1 &a_0\\
	\end{pmatrix}
	\]
	where the columns are indexed by the degree 1 elements in the order $v_{0,1}, \ldots, v_{d+1,1}$ and similarly for the rows with the degree 2 elements. This is a triangular matrix with a 
	zero on the diagonal, so it cannot have full rank regardless of what the values of the $a_i$ are.
	Therefore, the map is not injective and there is no weak Lefschetz element in $R_{\Delta_Q}$.
\end{proof}

Despite the non-existence of a weak Lefschetz element, the $h^*$-vectors of these simplices are easily computed and found to be of the form $(1,d+2,d+2,\ldots,d+2,1)$.
Thus, unimodality still holds for this family, indicating that unimodality, if it holds in general for integrally closed reflexive simplices, is a consequence of some subtle properties of these polytopes.

\bibliographystyle{plain}
\bibliography{Braun}

\end{document}